\documentclass[a4paper, 10pt]{amsart}

\usepackage{amsmath,amssymb}
\usepackage{array}
\usepackage{color}
\usepackage{hyperref}
\usepackage{longtable}
\newtheorem{theorem}{Theorem}[section]
\newtheorem{lemma}[theorem]{Lemma}
\newtheorem{corollary}[theorem]{Corollary}

\theoremstyle{definition}
\newtheorem{example}[theorem]{Example}
\newtheorem{definition}[theorem]{Definition}

\newcommand{\N}{\mathbb N}
\newcommand{\Z}{\mathbb Z}

\newcommand{\Q}{\mathbb Q}

\newcommand{\g}{\mathrm{g}}
\newcommand{\F}{\mathrm{F}}

\newcommand{\m}{\mathrm{m}}
\newcommand{\e}{\mathrm{e}}
\newcommand{\Ap}{\mathrm{Ap}}

\newcommand{\dsum}{\displaystyle\sum}

\usepackage{lscape}

\numberwithin{equation}{section}

\usepackage[ruled,vlined]{algorithm2e}
\usepackage{setspace}

\begin{document}
{
\address{Departamento de {\'A}lgebra, Universidad de Granada,
Granada 18071, Spain}

\email{vblanco@ugr.es}

\author{V{\'i}ctor Blanco}
\thanks{This work has been partially supported by the Junta de Andaluc\'ia grant number {\it P06--FQM--01366} and Ministerio de Ciencia e Innovaci\'on of Spain project Ref. {\it MTM2010-19576-C02-01}}


\begin{abstract}
In this paper we present a mathematical programming formulation for the $\omega$ invariant of a numerical semigroup for each of its minimal generators.
 The model consists of solving a problem of optimizing a linear function over the efficient set of a multiobjective linear integer program. We offer a methodology to solve this problem and we provide some computational experiments to show the applicability of the proposed algorithm.
\end{abstract}

\title[The omega invariant by Optimization Over an Eff. set]{On the computation of the Omega invariant of a numerical semigroup by optimizing over an efficient integer set}

\maketitle

\section{Introduction} \label{sec:1}

Applying discrete optimization usually means solving real-world problems in industry, transportation, location, etc, but there are many problems
arising in abstract mathematics, specially in commutative algebra, that can be formulated as integer programming problems and that have centered the attention of many researchers. In particular, problems in classical number theory are specially interesting and only  a few papers appear in the literature where operational researchers deal with these type of problems.

In the recent years, there have appeared several methods based on algebraic tools to solve  (single and multi-objective) discrete optimization problems as Gr\"obner bases \cite{blanco-puerto09, conti-traverso91} or short generating functions \cite{barvinok, blanco-puerto09a, deloera04, deloera09, lasserre}, amongst others. In this paper we intend to give an inverse
viewpoint, that is, how optimization tools can help to solve algebraic problems and then getting closer both mathematical fields, algebra and operational research.

Here, we study a specific problem in a basic algebraic structure: numerical semigroups. It is a simple framework  where developing
optimization tools to make computations that are usually done by brute force.

A numerical semigroup is a set of non-negative integers, closed under
addition, containing zero and such that its complement in $\N$ is finite. Details about the theory of numerical semigroups can be found in the recent monograph by Rosales and Garc\'ia-S\'anchez \cite{Ro-GS09}.

It can be checked that many of the main notions in the theory of semigroups are, in fact, defined or characterized as discrete optimization problems. For instance, for a numerical semigroup $S$,
its multiplicity, $\m(S)$, is the least integer belonging to $S$, that is
$$
\m(S) = \min\{ s \in S\}.
$$
Note that, if the reader is not familiar with numerical semigroups, may be it may not be possible to see why the above problem is a discrete optimization problem, but any numerical semigroup is minimally generated by a finite
set of integers $\{n_1, \ldots, n_p\}$ with $gcd(n_1, \ldots, n_p)=1$, and then, any element in $S$ can be expressed as a linear combination, with non-negative integer coefficients, of this set. Hence, the multiplicity of $S$ can be written as
$$
\m(S) = \min\{n_0: \dsum_{i=1}^p x_i\,n_i=n_0, x \in \Z_+^p, n_0\in \Z_+\}.
$$
In this case, it is clear that $\m(S)= \min\{n_1, \ldots, n_p\}$.

Other well-known index is the \textit{Frobenius number}. If $S$ is a numerical semigroup, its Frobenius number, $\F(S)$, is the largest integer not belonging to $S$. This index is well defined since by definition
$\N\backslash S$ is finite. Then,

$$
\F(S)=\max \{n \in \Z\backslash S\}
$$

Selmer proved that when a numerical semigroup, $S$, is generated by only two elements $a$ and $b$, then $\F(S)=ab - (a+b)$. However, there have not been many successes in finding formulas to compute the
Frobenius number when the numerical semigroup is generated by more than two integers. Unfortunately, this problem cannot be expressed as a linear integer problem.

The last classical index that we want to mention is the genus of a numerical semigroup. If $S$ is a numerical semigroup, the \textit{genus} of $S$, $\g(S)$, is the cardinal of $\N \backslash S$ (its number of holes), that is, $\g(S)=\# \N\backslash S$.

New arithmetic invariants for commutative semigroups, and then, in particular for numerical semigroups that have recently appeared in the literature are the tame degree, the catenary degree and the $\omega$ invariant (see \cite{AChKT10, B-GS-G10, C-G-L09, Ge-Ha08a, Ge-Ha08b, Ge-Ha-Le07,  Ge-Ka10a, Om10a} for a complete algebraic description of these indices). These arithmetic invariants come from the field of commutative algebra and factorization theory, and their definitions are not intuitive from a combinatorial viewpoint. In particular, the $\omega$ invariant allows to derive crucial finiteness properties in the theory of non-unique factorizations for semigroups (tameness). However, in \cite{B-GS-G10} the authors give characterizations
of some of these concepts that allow the optimization approach giving here.

    The goal of this paper is to present a new method for computing the $\omega$ invariant of a numerical semigroup by optimizing a linear function over
the non-dominated solutions of a multiobjective integer programming
problem. In contrast to usual integer programming problems, in multiobjective problems there
are several objective functions to be optimized. Usually, it is not possible to optimize all the objective functions simultaneously since
objective functions induce a partial order over the vectors in the feasible region, so a different
notion of solution is needed. A feasible vector is said to be a non-dominated (or Pareto optimal )
solution if no other feasible vector has component-wise larger objective values. The evaluation
through the objectives of a non-dominated solution is called efficient solution.

There are nowadays relatively few exact methods to solve general
multiobjective integer and linear problems (see \cite{ehrgott02}).
Some of them combine optimality of
the returned solutions with adaptability to a wide range of
problems ~\cite{zionts79, marcotte86}. Other methods, such as Dynamic Programming, are general
methods for solving, not very efficiently, general families of
optimization problems (see \cite{karwan-villareal82}). A different
approach, as the Two-Phase method (see \cite{ulungu-teghem95}),
looks for supported solutions (those that can be found as
solutions of a single-objective problem over the same feasible
region but with objective function a linear combination of the
original objectives) in a first stage and non-supported solutions
are found in a second phase using the supported ones. The
Two-phase method combines usual single-criteria methods with
specific multiobjective techniques.

Optimizing over an efficient set consists then of maximizing (or minimizing) a single objective function over the solutions of a multiobjective
problem. Since it is not known the structure of the solution set of that problem, different techniques that those used to
solve single or multiobjective problems are needed if we do not want to enumerate all the efficient solutions and then evaluate the function over all of them. Although there have appeared some papers analyzing
 this problem when the problem is continuous (see for instance, \cite{benson} or \cite{yamamoto}), only a few study the discrete case (\cite{abbas2006, jorge, nguyen92}). We apply an adaptation of the algorithm presented in \cite{jorge} to compute the $\omega$ invariant of a numerical semigroup. We choose that methodology since the approaches given in \cite{abbas2006} and \cite{nguyen92} does not fit the requirement of our problem.

 Although there is no need of computing the $\omega$ invariant of a numerical semigroup in a short time, the hardness of computing this index even for a small number of generators makes it difficult to analyze the invariant. Furthermore, algebraic researchers are showing an interest in studying certain combinatorial properties of non-unique
factorizations (in particular by the $\omega$ invariant), but this difficulty makes impossible to go further in the study of this index. In \cite{AChKT10} the authors give an algorithm to compute the $\omega$ invariant of a numerical semigroup by using techniques from factorization theory in monoids, and they give a short list of small problems that can be solved by using that methodology. We provide here a new method that avoid the complete enumeration of the non-dominated solutions of the subjacent multiobjective problem, and then, we are able to compute more efficiently the $\omega$ invariant of a numerical semigroup. This allows us to obtain computations that cannot be done up to the moment by using the brute force algorithm and then, better analysis of the behavior of this invariant. In the last section of the paper we show that our method, not only improve the CPU computation times, but the sizes of the problems that can be solved. However, in this paper we do not underestimate the algorithm in \cite{AChKT10}. Our algorithm seems to be able to compute the invariant for larger instances, but the approach in \cite{AChKT10} allows to give interesting conclusions for some special families of numerical semigroups. We center here in the computational side of the problem.
Other advantage of formulating the omega index as a mathematical programming problem is that a better analysis of the constraints and the objective function may improve the general running of the method.

In Section \ref{sec:2} we introduce the preliminary definitions and results for the rest of the paper. There, we describe the objects under study, the numerical semigroups, and the index we are interested to compute, the $\omega$-invariant.
An algorithm for optimizing a linear function over an efficient integer set  is described in Section \ref{sec:3}, and it allows to compute efficiently the $\omega$-invariant of a numerical semigroup. In Section \ref{sec:4} we show some computational tests done to
 check the efficiency of the algorithm. This paper is almost self-contained and written to be accessible for both algebraists and operational researchers.

\section{Preliminaries}
\label{sec:2}

We begin this section introducing the structure under the study of this paper. A numerical semigroup is a subset $S$ of $\N$ (here $\N$ denotes the set of non-negative integers) closed under
addition, containing zero and such that $\N\backslash S$ is finite. We say that $\{n_1, \ldots, n_p\}$ is a system of
generators of $S$ if $S=\{\dsum_{i=1}^p n_i x_i: x_i \in \N, i=1, \ldots, p\}$. We denote $S=\langle n_1, \ldots, n_p\rangle$ if $\{n_1, \ldots, n_p\}$
is a system of generators of $S$.  Note that any numerical semigroup has an unique minimal system of generators, in the sense that
no proper subset of it is a system of generators. Then, we assume through this paper that if $S=\langle n_1, \ldots, n_p\rangle$ then
$\{n_1, \ldots, n_p\}$ is the minimal system of generators of $S$ and $n_1 < \cdots < n_p$. The cardinality of the minimal system of generators, $p$, is called the \textit{embedding
dimension} of $S$.

Note that, for $S=\langle n_1, \ldots, n_p\rangle$ a numerical semigroup we can define a homomorphism $\pi: \N^p \rightarrow S$ as
$$
\pi(x_1, \ldots, x_p) = \dsum_{i=1}^p x_i\,n_i
$$
usually called \textit{the factorization homomorphism}.  For any element $s$ in $S$ we define $\mathsf Z (s)$ as the set $\pi^{-1}(s) =\{x \in \N^p: \dsum_{i=1}^p x_i n_i = s\}$. The \textit{factorization monoid} is the set $\mathsf Z (S) = \pi^{-1}(S)$.

For these semigroups we are interested in computing an index that measures how far away generators of a numerical semigroup are from primes: the $\omega$ invariant (see \cite{Ge-Ha-Le07,
Ge-Ha08a, Ge-Ha08b, Ge-Ka10a} for further details).

\begin{definition}
\label{defi:omega}
Let  $S=\langle n_1, \ldots, n_p\rangle$  be  a numerical semigroup. For $s \in S$, let \ $\omega (S, s)$ \ denote
the
      smallest \ $N \in \N_0 \cup \{\infty\} $ \ with the following
      property{\rm \,:}
      \begin{enumerate}
      \smallskip
      \item[] For all $n \in \N $ and $s_1, \ldots, s_n \in S $, if
              $\dsum_{i=1}^p s_i - s \in S$, then there exists a
              subset $\Omega \subset \{1, \ldots, n\} $ such that $|\Omega | \le N $ and
              \[
              \sum_{j \in \Omega} s_j - s \in S\,.
              \]
      \end{enumerate}
      Furthermore, we set
      \[
      \omega (S) = \max \{ \omega (S, n_i): i = 1, \ldots, p\} \in \N.\]
\end{definition}
Note that $\omega(S) <\infty$ since any numerical semigroup satisfies the ascending chain condition for $v$-ideals (see \cite{B-GS-G10} and the references therein for further details).

It is not direct to connect the above definition with optimization theory. However, the following result, proved in \cite{B-GS-G10}, is crucial to
develop the integer programming method proposed in this paper. Here, if $A \subseteq \N^p$, for some $p\leq 1$, we denote by \text{\rm Minimals} $A$ the set of minimals elements in $A$ with respect to the component-wise order in $\N^p$.
\begin{theorem}[\cite{B-GS-G10}]
\label{theo:omega}
Let $S=\langle n_1, \ldots, n_p\rangle$ be a numerical semigroup.
\begin{enumerate}
\item For every  $s \in S$, $\omega (S, n) = \max \{ \dsum_{i=1}^p x_i: x \in \text{\rm Minimals} \,\mathsf Z (n + S)\}$,
\item $\omega (S) = \max \{ \dsum_{i=1}^p x_i : x \in \text{\rm Minimals} \bigl(
      \mathsf Z (n_i + S) \bigr) \ \text{for some} \ i=1, \ldots, p \bigr\}$.
\end{enumerate}
\end{theorem}

We denote by $|x|=\dsum_{i=1}^p x_i$ the \textit{length} of $x\in\N^p$. From the above result, compute the $\omega$ invariant consists of maximizing the lengths of the minimal elements in the factorization monoid.

In the next section we connect the above result with the problem of optimizing a linear function over the efficient set of a multiobjective integer linear problem. In what follows, we recall the general structure of the optimization problems that are
involved in the mathematical programming formulation of the omega invariant. This part of the paper can be skipped if the reader is familiar with this theory.

An \textit{integer linear programming} problem consists of finding the largest (or smallest) value of a linear operator (objective function) $c: \Z^n_+ \rightarrow \Z$ when $c$ only takes values over a feasible region
defined by a system of linear equations:

\begin{eqnarray*}
\begin{split}
 a_{11}x_1 &+ a_{12}x_2 +& \cdots &+ a_{1n}x_n &= b_1\\
& & \ddots & &\\
a_{m1}x_1 &+ a_{m2}x_2 +& \cdots &+ a_{mn}x_n &= b_m\\
\end{split}
\end{eqnarray*}
where $a_{ij}, b_j \in \Q$ for all $i\in \{1, \ldots, n\}$ and $j\in \{1, \ldots, m\}$. If we denote by $A=(a_{ij}) \in \Q^{m\times n}$ and by $b=(b_j) \in \Q^m$, this problem can be
written as

\begin{equation}
 \label{eq:ip}
\begin{array}{lll}
 \max \; c(x)&&\\
s.t.&&\\
& Ax&=b\\
& x &\in \Z^n_+
\end{array}\tag{\rm IP}
\end{equation}

It is well-known \cite{schrijver} that the complexity of solving the above problem is NP-complete, and difficult to solve even for a small number of variables (see \cite{cornuejols}). However, when the constraints ($Ax=b$) have a
concrete structure, or the formulation comes from a problem where more information is known, it is possible to develop exact (or heuristic with high precision) algorithms that allow to solve efficiently these problems (see \cite{nw}).
Recently, some papers have appeared where algebraic tools are used to solve problem \eqref{eq:ip} (see \cite{conti-traverso91, deloera04, lasserre}).

If we consider in \eqref{eq:ip} a linear operator $C: \Z^n_+ \rightarrow \Z^k$ (with $k\ge 2$) we get a integer linear multiobjective optimization problem:
\begin{equation}
 \label{eq:mip}
\begin{array}{lll}
 \max (\text{ or } \min)\; C(x)&&\\
s.t.&&\\
& Ax&=b\\
& x \in \Z^n_+
\end{array}\tag{\rm MIP}
\end{equation}

Since, in this case, objective function defines a non-total order over the set of solutions of $Ax=b$ with $x \in \Z^n_+$, it is needed to define what we understand about solving \eqref{eq:mip}. We say that a vector $x\in \Z^n_+$ with $Ax=b$ is a \textit{non-dominated (or Pareto optimal) solution} if there is no other vector $y \in \Z^n_+$ with
$Ay=b$ and such that $C(y) \leq C(x)$ and $C(x)\neq C(y)$ where $\leq$ denotes the component-wise order in $\Z^k$. Then, the solution of \eqref{eq:mip} is a set of feasible solutions that are not comparable and such that we cannot find better solutions with respect to the partial ordering induced by $C$. In the case we want to minimize instead of maximize, the ordering is adapted conveniently.

In general, the complexity of solving multiobjective integer problems is $\sharp$P-hard and just a few algorithms appear in the literature to solve exactly this kind of problems. The recent works by Blanco and Puerto \cite{blanco-puerto09}, De Loera et. al \cite{deloera09} and Blanco and Puerto \cite{blanco-puerto09a} present algebraic approaches for solving these problems. An interesting monograph for multiobjective optimization is \cite{ehrgott02}.

In the next and last step we recall the goal of optimizing over an efficient integer region. Note that given the set of non-dominated solutions of \eqref{eq:mip} we may be interested in finding
 those non-dominated solutions that maximize a linear operator. Then, we are interested in solving the following optimization problem:
\begin{equation}
 \label{eq:oes}
\begin{array}{lll}
 \max \; c(x)&&\\
s.t.& \text{$x$ is a non-dominated solution of \eqref{eq:mip}}
\end{array}\tag{\rm OES}
\end{equation}
It is not possible, in general, to express the constraint of the above problem as a linear system of equations. Actually, as far as we know, there are no descriptions for the set of solutions of a multiobjective problem, and then, to solve \eqref{eq:oes} one may think of
solving first the multiobjective problem, and then, selecting those solutions with largest values of $c$. This process is not effective and it makes it difficult when the efficient set has a large size.

However, there exists a few algorithms to solve \eqref{eq:oes} without enumerating the complete set of non-dominated solutions of \eqref{eq:mip}. In particular, we will adapt the algorithm presented in \cite{jorge} to solve the problem of optimizing over the efficient integer set related to the $\omega$-invariant. The following result states the structure of the optimization problem that solves the $\omega$-invariant.

\begin{theorem}
\label{theo:1}
Let $S=\langle n_1, \ldots, n_p\rangle$ be a numerical semigroup. Then, for each $j\in \{1, \ldots, p\}$, $\omega(S, n_j)$ is the solution of the following OES problem:
\begin{equation}
\label{eq:omega}
\begin{array}{lll}
\max \;c(x)=\dsum_{i=1}^n x_i& & \\
s.t.& &\\
&x \in \min C(x)=(x_1, \ldots, x_p)&\\
& s.t. &\\
& \dsum_{i=1}^p x_i n_i - \dsum_{i=1}^p y_i n_i = n_j&\\
& x, y \in \Z_+^p&
\end{array}\tag{${\rm OES}_j^0$}
\end{equation}
\end{theorem}
\begin{proof}
Observe that the elements in $\text{\rm Minimals} \bigl(
      \mathsf Z (n_j + S) \bigr)$ are those minimal elements with respect to the component wise ordering in $\mathsf Z (n_j + S)$, and $\mathsf Z (n_j + S)= \{(x_1, \ldots, x_p): \dsum_{i=1}^p n_i\,x_i \in n_j + s, \text{ for some } s \in S\} = \{(x_1, \ldots, x_p): \dsum_{i=1}^p n_i\,x_i = n_j + \dsum_{i=1}^p n_i\,y_i \text{  for some } y_i\in \N\}$.
\end{proof}
Problem \eqref{eq:omega} consists of maximizing the lengths of the non-dominated solutions of a multiobjective problem. We denote by $({\rm SMIP}_j)$ that subjacent multiobjective problem

From the formulation of the $\omega$-invariant given in the above result, it is clear that if we denote by $\e_j$, the $p$-tuple having a $1$ as its $j$th entry and zeros otherwise, for each $j \in \{1, \ldots, p\}$, $\e_j \in \mathsf Z (n_j + S)$ since we can always write $n_j = \dsum_{i\neq j} 0\cdot n_i + 1\cdot n_j$ and $(0, \ldots, 0)$ is not a feasible solution. Consequently, if the multiobjective subproblem has more non-dominated solutions they must have $x_j=0$ to be not comparable with $\e_j$. The next lemma states that these unit vector cannot be solution of the multiobjective subjacent  problem.
\begin{lemma}
\label{lemma:1}
Let $S=\langle n_1, \ldots, n_p\rangle$ be a numerical semigroup. Then, $\omega(S, n_j) > 1$ for all $j \in \{1, \ldots, p\}$.
\end{lemma}
\begin{proof}
From Definition \ref{defi:omega}, $\omega(S, n_j)\ge 1$. Let us see that $\omega(S, n_j)\neq 1$. Assume that $\omega(S, n_j)=1$, then, if $\dsum_{i=1}^n s_i - n_j \in S$, there exists some $s_k$ such that $s_k-n_j \in S$. That is, if $n_j= \dsum_{i=1}^n s_i + s$ for some $s \in S$, then $n_j = s_k + s'$ for some $s' \in S$, for any $s_1, \ldots, s_n$ and any $n \in \N$. It is clear that it is not possible since that assumption is the same as stating that the generators of $S$ are prime, but there are no primes in a numerical semigroup (see \cite{Ro-GS09}).
\end{proof}

Furthermore, note that the feasible region in \eqref{eq:omega} is not bounded. However, by solving linear integer problems we can obtain upper bounds for each of the variables $x_i$ involved in the problem, making bounded the feasible region. For each $j \in \{1, \ldots, p\}$ and $k\neq j$, let $ub_{jk}$ the optimal value of the following integer program:
\begin{equation}
\label{eq:ub}
\begin{array}{lll}
\min \; x_j& & \\
& s.t. &\\
& x_j n_j - \dsum_{i=1}^p y_i n_i = n_k&\\
& y_k=0\\
& x_j \in \Z_+, y \in \Z_+^p&
\end{array}\tag{${\rm UB}_{jk}$}
\end{equation}
The solutions of \eqref{eq:ub} are those feasible solutions of \eqref{eq:oes} in the form $x_j\,e_j$. An upper bound for each variable in \eqref{eq:omega} is then $ub_j = \max_k \{ub_{jk}\}$.

Note that if $x$ is a non-dominated solution of the multiobjective subproblem in \eqref{eq:oes}, and $x_j\neq 0$, then $y_j=0$, and also, if $y_j\neq 0$, then $x_j=0$. It is clear since if $x_j$ and $y_j$ have nonzero values, then, we can always find a feasible solution $(x', y')$ where $x'_j=\max\{0, x_j-y_j\}$, $y'_j=\max\{0, y_j-x_j\}$, and $x'_i=x_i$ and $y'_i=y_i$ for $i\neq j$. This result is also proved in \cite{B-GS-G10} from an algebraic viewpoint, in terms of the supports of $x$ and $y$.

As a direct consequence of Lemma \ref{lemma:1}, Theorem \ref{theo:1} and the comment above we get the following result:
\begin{corollary}
\label{corollary:1}
Let $S=\langle n_1, \ldots, n_p\rangle$ be a numerical semigroup. Then, for each $j\in \{1, \ldots, p\}$, $\omega(S, n_j)$ is the solution of the following OES problem:
\begin{equation}
\label{eq:oesf}
\begin{array}{lll}
\max \;\dsum_{i=1}^n x_i& & \\
s.t.& &\\
&x \in \min (x_1, \ldots, x_p)&\\
& s.t. &\\
& \dsum_{i=1}^p x_i n_i - \dsum_{i=1}^p y_i n_i = n_j&\\
& x_i \leq ub_i\\
& x_j =0\\
& x, y \in \Z_+^p&
\end{array}\tag{${\rm OES}_j$}
\end{equation}
\end{corollary}

Now, computing the omega invariant consists of solving a problem of optimizing a linear function over the efficient region of a bounded integer linear multiobjective problem. In the next section we describe a methodology for solving \eqref{eq:oesf}.

\section{An algorithm for computing $\omega(S,n_j)$}
\label{sec:3}

In this section we adapt the procedure given in \cite{jorge} for solving \eqref{eq:oesf}, and consequently compute the omega invariant of a numerical semigroup.
For the sake of that, we first give some previous results related to each of the steps of the algorithm.

Let $j \in \{1, \ldots, p\}$. The first step of the proposed procedure consists of stating that we effectively need to design an algorithm to compute $\omega(S, n_j)$. With this aim, we consider a relaxed problem of \eqref{eq:oesf} consisting of obviating that the feasible solutions are just the minimal non-dominated solutions of the multiobjective problem, but the feasible solutions of $({\rm SMIP}_j)$. This problem is the following
\begin{equation}
\label{eq:R}
\begin{array}{lll}
\max \;\dsum_{i=1}^n x_i& & \\
s.t.& &\\
& \dsum_{i=1}^p x_i n_i - \dsum_{i=1}^p y_i n_i = n_j&\\
& x_i \leq ub_i\\
& x_j =0\\
& x, y \in \Z_+^p.&
\end{array}\tag{${\rm R}_j$}
\end{equation}
If we solve \eqref{eq:R} and the optimal solution is a non-dominated solution of $({\rm SMIP}_j)$, then, that solution is the solution of \eqref{eq:oesf}. Furthermore, if \eqref{eq:R} is infeasible, \eqref{eq:oesf} has no solution. The next result gives us some information about the above problem.

\begin{lemma}
\label{lemma:3}
Problem \eqref{eq:R} is feasible. Furthermore, the optimal solutions of \eqref{eq:R} are not non-dominated solution of
$({\rm SMIP}_j)$.
\end{lemma}

\begin{proof}
By the definition of the $\omega$ invariant and the characterization given in Theorem \ref{theo:1}, it is clear that \eqref{eq:R} cannot be infeasible. Moreover, the problem is bounded due to the finiteness of $\omega(S)$ (see comment after Definition \ref{defi:omega}).

Let $(x^*, y^*)$ an optimal solution of \eqref{eq:R}. Assume first that $x^*_i=ub_i>0$ for all $i=1, \ldots, p$ with $i\neq j$ ($x^*_j=0$). Then, by the comment done before Corollary  \ref{corollary:1}, $y^*_i=0$ for all $i\neq j$, and then, by feasibility $\sum_{i\neq j} x^*_i\,n_i=n_j$, that is not possible since we are supposing that $\{n_1, \ldots, n_p\}$ is a minimal system of generators of $S$. Then, there exists at least one $k\neq j$ such that $x^*_k<ub_k$.

Assume now that $(x^*, y^*)$ is a non-dominated solution of $({\rm SMIP}_j)$, and let $\overline{x}=x^*+(ub_k-x^*_k)\,\e_k$ and $\overline{y}= y^* + (ub_k-x^*_k)\e_k$. By definition $(\overline{x}, \overline{y})$ is a feasible solution of $({\rm SMIP}_j)$ and $\overline{x}>x^*$, and consequently $\sum_{i=1}^p \overline{x} > \sum_{i=1}^p x^*$ contradicting the optimality of $x^*$.
 \end{proof}

The above result justifies the development of an algorithm that allows to improve the optimal solution of \eqref{eq:R} to obtain the optimal value of \eqref{eq:oesf}. In the next step we search for feasible solutions of
$({\rm SMIP}_j)$ that are non-dominated and that dominate the optimal solution of \eqref{eq:R}. For that, we use the following adaptation of the result by Ecker and Kouada \cite{ek}.
\begin{lemma}
\label{lemma:4}
Let $x^*$ be an optimal solution of \eqref{eq:R} and $(\overline{s}, \overline{x}, \overline{y})$ an optimal solution of the following problem
\begin{equation}
\label{eq:ek}
\begin{array}{lll}
\max \;\dsum_{i=1}^n x_i& & \\
s.t.& &\\
& x_i + s_i = x_i^* & i=1, \ldots, p\\
& \dsum_{i=1}^p x_i n_i - \dsum_{i=1}^p y_i n_i = n_j&\\
& x_i \leq ub_i & i=1, \ldots, p\\
& x_j =0\\
& x, y \in \Z_+^p.&
\end{array}\tag{${\rm EK}_j(x^*)$}
\end{equation}
Then, $\overline{x}$ is a non-dominated solution of $({\rm SMIP}_j)$ that dominates $x^*$.
\end{lemma}

With this lemma, we are able to compute non-dominated solutions of $({\rm SMIP}_j)$ from the solutions of \eqref{eq:R}. However, these solutions may not be the ones that maximize the lengths in \eqref{eq:oesf}.

To move through feasible solutions of our problem to obtain the wished solutions, we use the idea presented in \cite{jorge}, where the author gives a procedure to generate new non-dominated solutions. This procedure uses a result given in the book by Nemhauser and Wolsey \cite{nw}.

\begin{lemma}
\label{lemma:5}
Let $\overline{x}$ be a non-dominated solution ${(\rm{SMIP}_j)}$ and $(\hat{x}, \hat{y})$ an optimal solution of the following problem
\begin{equation}
\label{eq:nw}
\begin{array}{lll}
\max \;\dsum_{i=1}^n x_i& & \\
s.t.& &\\
& x_i \leq z_i (\overline{x}_i-1) - M_i(z_i-1) & i=1, \ldots, p\\
& \dsum_{i=1}^p x_i n_i - \dsum_{i=1}^p y_i n_i = n_j&\\
& x_i \leq ub_i & i=1, \ldots, p\\
& x_j =0\\
& \dsum_{i=1}^p z_i \ge 1&\\
& x, y \in \Z_+^p, z \in \{0,1\}^p&
\end{array}\tag{${\rm NW}_j(\overline{x})$}
\end{equation}
and where $M_j = \max\{x_j: n_jx_j = \dsum_{i\neq j}^p n_i\,y_i, x_i \leq ub_i, i=1, \ldots, p, x, y \in \Z_+^p\}$.

Then, $\hat{x}$ is a feasible solution of ${(\rm{SMIP}_j)}$ that is not dominated by $\overline{x}$.
\end{lemma}

By using the above lemma, in case the solution $\overline{x}$ is not a solution of \eqref{eq:oesf}, we can generate other feasible solution, non comparable with it and such that we can check if it is the one that we are looking for or, if it is not, continue iterating to reach it.

Therefore, if the solution $\overline{x}$ obtained is not optimal for \eqref{eq:oesf} and we find $\hat{x}$ by using Lemma \ref{lemma:5}, $x^*$ can be replaced by $\hat{x}$ in the procedure of Lemma \ref{lemma:4} to obtain a non-dominated solution that dominates $\hat{x}$. If this new non-dominated solution is not the optimal solution of \eqref{eq:oesf} we use again Lemma \ref{lemma:5}  to obtain a non comparable solution with this. But it may occur that this process cycle obtaining $\hat{x}$ again. However, this can be avoided by applying the following extension of Lemma \ref{lemma:5}, and also the reader can find it in \cite{nw}, that allows to obtain a feasible solution of $({\rm SMIP}_j)$ that is not comparable with a given list of non-dominated solutions $\overline{x}^1, \ldots, \overline{x}^s$.
\begin{lemma}
\label{lemma:6}
Let $\overline{x}^1, \ldots, \overline{x}^s$ be non-dominated solutions of  ${(\rm{SMIP}_j)}$ and $(\hat{x}, \hat{y})$ an optimal solution of the following problem
\begin{equation}
\label{eq:enw}
\begin{array}{lll}
\max \;\dsum_{i=1}^n x_i& & \\
s.t.& &\\
& x_i \leq z^k_i (\overline{x}^k_i-1) - M_i(z_i^k-1) & i=1, \ldots, p, k=1, \ldots, s\\
& \dsum_{i=1}^p x_i n_i - \dsum_{i=1}^p y_i n_i = n_j&\\
& x_i \leq ub_i & i=1, \ldots, p\\
& x_j =0\\
& \dsum_{i=1}^p z^k_i \ge 1&k=1,  \ldots, s\\
& x, y \in \Z_+^p, z^k \in \{0,1\}^p& k=1, \ldots, s
\end{array}\tag{${\rm NW}_j(\overline{x}^1, \ldots, \overline{x}^s)$}
\end{equation}
and where $M_j = \max\{x_j: n_jx_j = \dsum_{i\neq j}^p n_i\,y_i, x_i \leq ub_i, i=1, \ldots, p, x, y \in \Z_+^p\}$.

Then, $\hat{x}$ is a feasible solution of ${(\rm{SMIP}_j)}$ that is not dominated by $\overline{x}^1, \ldots, \overline{x}^s$.
\end{lemma}

With this method we are able to find a feasible solution that is not comparable with any of the solutions found in the previous iterations (those that we know that are not optimal for \eqref{eq:oes}). We can repeat the process until we get the desired solution, but when do we know that we have got the solution of our problem? To ask this question we analyze the last step of the described procedure.

Assume we have obtained a non-dominated solution of ${\rm SMIP}_j$, $\hat{x}$. If we solve \eqref{eq:enw} and we get $\tilde{x}$ as optimal solution with $\sum_{i=1}^p \tilde{x}_i < \sum_{i=1}^p \hat{x}_i$ then, $\hat{x}$ is the solution to our problem since among all the possible non-dominated solutions, the one that has maximum length has smaller length than $\hat{x}$. Moreover, if \eqref{eq:enw} is infeasible, clearly $\hat{x}$ is our solution.

The general procedure is describe in the pseudo-code of Algorithm \ref{alg:1}

{\small
\begin{algorithm}[h]
\label{alg:1}
\SetKwInOut{Input}{Input}
\SetKwInOut{Output}{Output}
\SetKwInOut{Initialization}{Initialization}
\Input{A minimal generating system of a numerical semigroup $S$: $\{n_1, \ldots, n_p\}$ and $j \in \{1, \ldots, p\}$.}
$l:=1$

Solve \eqref{eq:R}: $x^1$.

\While{$\omega(S, n_j)$ is not obtained}{
\begin{itemize}
\item Set $x^*=x^l$ and solve \eqref{eq:ek}: $\hat{x}^l$.
\item Solve \eqref{eq:enw}: $\overline{x}^l$.
\end{itemize}
\eIf{\eqref{eq:enw} is infeasible or $\sum_{i=1}^p \overline{x}^l_i < \sum_{i=1}^p \hat{x}^l_i$}{\noindent$\hat{x}^l=\omega(S, n_j)$}{\noindent Set $x^{l+1}=\hat{x}$ and $l:=l+1$}
}

\Output{$\omega(S, n_j)$.}
\caption{Computing the omega invariant by optimizing over an efficient set.}
\end{algorithm}}

A direct consequence of all the above comments about the proposed procedure is the following.

\begin{theorem}
Let $S=\langle n_1, \ldots, n_p\langle$ be a numerical semigroup. Algorithm \ref{alg:1} computes $\omega(S, n_j)$ in a finite number of steps. As $\omega(S)=\max \{\omega(S, n_j): j \in 1, \ldots, p\}$, Algorithm \ref{alg:1} also allows the computation of $\omega(S)$.
\end{theorem}

The following example illustrates the usage of Algorithm \ref{alg:1}.

\begin{example}
Let $S=\langle 6, 13, 14 \rangle$. The corresponding bounds are $ub_1=9$, $ub_2=2$ and $ub_3=4$ and the constants for \eqref{eq:enw} are $M_1=9$, $M_2=8$ and $M_3=9$.
The optimal solution for ${\rm R}_1$ is $x^1=(0,9,9)$. Solving now $({\rm EK}(x^1))$ we obtain $\hat{x}^1=(0,2,0)$, and finally, the optimal solution of ${\rm NW}(\hat{x}^1)$ is $\overline{x}=(0,1,9)$. Since $\sum_{i=1}^3 \overline{x}_i = 10 > \sum_{i=1}^3 \hat{x}^l_i = 2$, we continue with the algorithm with $x^2=(0,1,9)$.

Computing now the solution of $({\rm EK}(x^1))$ we get $\hat{x}^2=(0,0,3)$ and ${\rm NW}(\hat{x}^1, \hat{x}^2)$ is infeasible, so $\omega(S, 6)=3$.

Analogously we obtain $\omega(S, 13)=9$ and $\omega(S, 14)=7$, being $\omega(S)=\max\{3, 7, 9\} = 9$.
\end{example}
Note that further improvements can be done in the above algorithm by incorporating bounds for the $\omega$ invariant at each step of it. For instance, we can initially set $u_0=0$ and $v_0=\dsum_{i=1}^p x^1_i$, the lower and upper bounds for $\omega(S, n_j)$, respectively. Then, update $u_0$ with the optimal value of \eqref{eq:ek} and $v_0$ with the optimal value of \eqref{eq:enw} when they are improved. With this, we have that either when $u_0=v_0$ or the optimal value of \eqref{eq:enw} is smaller than $u_0$, the solution is the last solution obtained by solving \eqref{eq:ek}. We have included these bounds to our implementation for the computational results, although some others could be added by incorporating these bounds conveniently to the constraints of the problems.

Moreover, if the Apery sets of the numerical semigroup are known, the following lemma, that appears in \cite{B-GS-G10}, helps us to simplify the problem of searching minimal solutions.
\begin{lemma}[\cite{B-GS-G10}]
Let $S = \langle n_1, \ldots, n_p\rangle$ be a numerical semigroup, $j \in \{1, \ldots, p\}$, and $(x^*, y^*) \in  ({\rm SMIP}_j)$. Then, if $x^*_i\neq 0$, $\dsum_{k=1}^p n_k\,y_k^* \in Ap(S, n_i)$.
\end{lemma}
\begin{proof}
Choose $i$ such that $x^*_i\neq 0$. Since $x^*$ is feasible for $({\rm SMIP}_j)$, $\dsum_{k=1}^pn_k\,x_k^*=n_j+\dsum_{k=1}^p n_k\,y_k^*$. From the minimality of $(x^*, y^*)$, $x^*-\e_j$ cannot be minimal, that is, $\dsum_{k=1}^pn_k\,x_k^*-n_i$ cannot be expressed as a integer linear combination of $n_1, \ldots, n_p$, or equivalently, $\dsum_{k=1}^pn_k\,x_k^*-n_j-n_i=n_j+\dsum_{k=1}^p n_k\,y_k^*-n_j\not\in S$. This implies that $\dsum_{k=1}^p n_k\,y_k^*\in \Ap(S,n_i)$.
\end{proof}
Then , with the above result, the equation defining the feasible region of $ ({\rm SMIP}_j)$ can be written as
\begin{equation}
\dsum_{k=1}^pn_k\,x_k=n_j+ w \qquad w \in \bigcup_{i\neq j} \Ap(S, n_i)
\end{equation}

Since the above constraint is not linear, we can add part of the information contained on it by incorporating inequalities in the form $\sum_{i=1}^p n_i\,y_i \leq \max \bigcup_{i=1}^p \Ap(S, n_i)$ or $\sum_{i=1}^p n_i\,y_i \geq \min \bigcup_{i=1}^p \Ap(S, n_i) = \min\{n_1, \ldots, n_p\}$
\section{Computational experiments}
\label{sec:4}
In this section we present the results of some computational experiments done to analyze the applicability of the algorithm proposed in Section \ref{sec:3}. For the sake of that we have generated a set of instances of numerical semigroups with different numbers of minimal generators. From the computational viewpoint, up to the moment, we can only compare our algorithm with a recent implementation done in GAP \cite{numericalsgps} for \cite{B-GS-G10} (available upon request), and that consists of enumerating all the feasible solutions of $({\rm SMIP}_j)$, comparing them to determine which are the minimal solutions, and finally selecting those with largest length.

Our algorithm has been implemented in XPRESS-Mosel 7.0 that allows to solve the single-objective integer problems involved in the resolution of the $\omega$ invariant, by using a branch-and-bound method and nesting models by calling the library \texttt{mmjobs}. The algorithms have been executed on a PC with an
Intel Core 2 Quad processor at 2x 2.50 Ghz and 4 GB of RAM.

It is clear that the complexity of the algorithm depends of the dimension of the space (embedding dimension of the numerical semigroups) and the size of the coefficient of the constraints. Then, we have generated for each $p \in \{5, 10, 15, 20\}$, five instances of $p$ non-negative integers ranging in $[2, 1000]$. Note that larger instances are difficult to find since even for $p=20$, selecting $20$ non-negative integers with great common divisor equal to one and such that they are a minimal system of generators of a numerical semigroup is not an easy task. We have used recursively the function \texttt{RandomListForNS} of GAP until we found the list of integer with the above requirements. In Table \ref{table:1} we show the battery of numerical semigroup for which we have run the algorithms. In the first column we write the name we give to the semigroup, in the second column the set of minimal generating system of the numerical semigroup, and in the last column the embedding dimension of the semigroup.

\begin{longtable}{|c|>{\scriptsize}c|c|}\hline
\texttt{name} & \texttt{minimal generating system} & \texttt{embdim}\\\hline
\texttt{S}$5(1)$ & $\{20,354,402,417,429\}$ & \\
\texttt{S}$5(2)$ & $\{7,292,359,645,755\}$ & \\
\texttt{S}$5(3)$ & $\{5,86,99,148,152\}$ & $5$\\
\texttt{S}$5(4)$ & $\{41,65,155,317,377\}$ & \\
\texttt{S}$5(5)$ & $\{28,55,125,233,590\}$ & \\\hline
\texttt{S}$10(1)$ & $\{43,63,68,108,120,135,142,150,177,224\}$& \\
\texttt{S}$10(2)$ & $\{15,46,58,89,108,114,117,126,130,173\}$& \\
\texttt{S}$10(3)$ & $\{20,22,24,26,54,77,83,89,93,95\}$& $10$\\
\texttt{S}$10(4)$ & $\{96,131,136,171,173,239,278,287,364,483\}$& \\
\texttt{S}$10(5)$ & $\{146,173,207,359,426,548,604,606,657,702\}$& \\\hline\newpage\hline
\texttt{S}$15(1)$ & $\{36,47,65,79,82,84,91,96,100,109,121,124,134,139,169\}$& \\
\texttt{S}$15(2)$ & $\{46,115,155,286,289,341,342,348,393,436,445,449,504,527,584\}$& \\
\texttt{S}$15(3)$ & $\{40,84,126,130,132,135,142,152,165,183,217,221,229,273,323\}$ & $15$\\
\texttt{S}$15(4)$ & $\{75, 104, 114, 128, 216, 219, 241, 271, 309, 310, 321, 327, 340, 352, 371\}$\\
\texttt{S}$15(5)$ & $\{29, 50, 95, 96, 99, 109, 110, 119, 131, 134, 135, 152, 162, 180, 201\}$ &\\\hline
\texttt{S}$20(1)$ &  $\{131,145,249,257,260,319,354,459,465,469,487,572,575,587,606,607,652,674,694,762\}$&\\
\texttt{S}$20(2)$ & $\{ 57, 105, 182, 186, 201, 204, 254, 259, 263, 274, 275, 294, 295, 298, 307, 338, 367, 393, 417, 431\}$&\\
\texttt{S}$20(3)$ & $\{ 85, 298, 333, 342, 349, 358, 401, 415, 462, 480, 556, 569, 583, 609, 619,  708, 710, 752, 821, 853\}$& $20$\\
\texttt{S}$20(4)$ &  $\{81, 107, 168, 194, 230, 236, 274, 277, 286,  290, 305, 310, 348, 351, 366, 379, 396, 416, 521, 583\}$&\\
\texttt{S}$20(5)$ & $\{101, 141, 279, 314, 329, 369, 399, 425, 438, 447, 477, 501, 534, 536, 555, 574, 620, 727, 786, 871 \}$ &\\ \hline

\caption{Battery of numerical semigroups for the computational experiments.}\label{table:1}
\end{longtable}

In tables \ref{table:2} and \ref{table:3} we summarize the obtained results. In the first column we write the numerical semigroup ($S$). The second column is
 the generator with respect to the $\omega$ invariant is computed ($\texttt{n}_j$). The value of $\omega(\texttt{S}, n_j)$ is shown in the third column (in this same column we mark  in bold face, 
for each numerical semigroup, the maximum value that coincides with $\omega(\texttt{S})$). One of the minimal vectors that reach the 
maximum length is shown in the fourth column. The next column, the fifth, shows the number of iterations of our algorithm ($\texttt{it}$). The CPU time (in seconds) needed to 
compute $\omega(S, n_j)$ with our algorithm  (\texttt{time}$_j$)  and with GAP (\texttt{GAPtime} -  $\ast$ indicates that GAP was not able to solve the problem
 in less than $2$ hours) is shown in columns sixth and seventh, respectively. 
The eigth shows the total  CPU time needed to compute $\omega(S)$ (\texttt{totaltime}). The average CPU time to compute each of the 
$\omega(S, n_j)$ (\texttt{avtime}) is written in column nineth. The last column is the number of minimal (nondominated) elements of the corresponding problem that was computed with GAP  ($\ast$ indicates that GAP was not able to compute this number in less than two hours).

For the computations of the experiments for $p=15, 20$ we add to the constraints of the multiobjective problem the inequalities $\sum_{i=1}^p n_i\,y_i \leq \max \bigcup_{i=1}^p \Ap(S, n_i)$ or $\sum_{i=1}^p n_i\,y_i \geq \min \bigcup_{i=1}^p \Ap(S, n_i) = \min\{n_1, \ldots, n_p\}$.
{\small
  \centering
    \begin{longtable}{|c|c|c|>{\small}c|c|c|c|c|c|c|}
\hline
    \texttt{S}     & $\texttt{n}_j$  & $\omega(\texttt{S}, \texttt{n}_j)$ & \texttt{min}   & \texttt{it}  & $\texttt{time}_j$    & \texttt{GAPtime} & \texttt{tottime} & \texttt{avtime} & $\#\texttt{min}$\\\hline
          & 20    & 4     & [0,0,0,0,4] & 9     & 0.54  & 6.03  &       &       & 12 \\
          & 354   & 60    & [60,0,0,0,0] & 12    & 0.95  & 11.35 &       &       & 14 \\
    S5(1) & 402   & 63    & [63,0,0,0,0] & 16    & 1.439 & 12.54 & 5.921 & 1.184 & 17 \\
          & 417   & 60    & [60,0,0,0,0] & 15    & 1.43  & 12.68 &       &       & 16 \\
          & 429   & 60    & [60,0,0,0,0] & 17    & 1.55  & 12.43 &       &       & 20 \\\hline
          & 7     & 3     & [0,3,0,0,0] & 10    & 0.55  & 12.48 &       &       & 11 \\
          & 292   & 93    & [93,0,0,0,0] & 9     & 0.37  & 23.72 &       &       & 11 \\
    S5(2) & 359   & 93    & [93,0,0,0,0] & 11    & 0.43  & 27.33 & 2.84  & 0.56  & 13 \\
          & 645   & 200   & [200,0,0,0,0] & 14    & 0.67  & 45.92 &       &       & 15 \\
          & 755   & 200   & [200,0,0,0,0] & 18    & 0.81  & 75.59 &       &       & 19 \\\hline
          & 5     & 2     & [0,0,0,2,0] & 8     & 0.285 & 1.201 &       &       & 11 \\
          & 86    & 37    & [37,0,0,0,0] & 11    & 0.294 & 2.527 &       &       & 12 \\
    S5(3) & 99    & 37    & [37,0,0,0,0] & 11    & 0.34  & 2.82  & 1.69  & 0.33  & 12 \\
          & 148   & 60    & [60,0,0,0,0] & 12    & 0.37  & 4.1   &       &       & 13 \\
          & 152   & 60    & [60,0,0,0,0] & 12    & 0.39  & 2.29  &       &       & 13 \\\hline
          & 41    & 14    & [0,14,0,0,0] & 12    & 0.893 & 5.64  &       &       & 14 \\
          & 65    & 22    & [22,0,0,0,0] & 13    & 0.988 & 6.02  &       &       & 14 \\
    S5(4) & 155   & 24    & [24,0,0,0,0] & 16    & 1.1   & 8.22  & 7.39  & 1.47  & 18 \\
          & 317   & 22    & [21,0,1,0,0] & 22    & 2.916 & 13.96 &       &       & 28 \\
          & 377   & 31    & [31,0,0,0,0] & 18    & 1.49  & 18.7  &       &       & 35 \\\hline
          & 28    & 10    & [0,10,0,0,0] & 11    & 0.5   & 10.71 &       &       & 12 \\
          & 55    & 25    & [25,0,0,0,0] & 8     & 0.381 & 11.45 &       &       & 12 \\
    S5(5) & 125   & 27    & [27,0,0,0,0] & 13    & 0.71  & 20.18 & 4.719 & 0.94  & 15 \\
          & 233   & 26    & [26,0,0,0,0] & 13    & 0.732 & 42.37 &       &       & 17 \\
          & 590   & 30    & [24,5,0,1,0] & 23    & 2.38  & 109.38 &       &       & 48 \\\hline
          & 43    & 5     & [0,0,5,0,0,0,0,0,0,0] & 49    & 3.36  & 5.41  &       &       & 58 \\
          & 63    & 8     & [8,0,0,0,0,0,0,0,0,0] & 48    & 2.7   & 8.61  &       &       & 65 \\
          & 68    & 8     & [8,0,0,0,0,0,0,0,0,0] & 49    & 3.16  & 13.18 &       &       & 69 \\
          & 108   & 7     & [5,0,2,0,0,0,0,0,0,0] & 52    & 4.15  & 18.26 &       &       & 81 \\
          & 120   & 8     & [8,0,0,0,0,0,0,0,0,0] & 57    & 4.24  & 12.65 &       &       & 94 \\
    S10(1) & 135   & 9     & [9,0,0,0,0,0,0,0,0,0] & 68    & 5.95  & 15.5  & 67.89 & 6.78  & 108 \\
          & 142   & 9     & [9,0,0,0,0,0,0,0,0,0] & 88    & 9.24  & 16.75 &       &       & 125 \\
          & 150   & 7     & [4,2,1,0,0,0,0,0,0,0] & 66    & 8.07  & 19.85 &       &       & 116 \\
          & 177   & 9     & [7,2,0,0,0,0,0,0,0,0] & 70    & 6.88  & 49.26 &       &       & 149 \\
          & 224   & 9     & [7,0,2,0,0,0,0,0,0,0] & 113   & 20.1  & 65.16 &       &       & 246 \\\hline
          & 15    & 3     & [0,0,3,0,0,0,0,0,0,0] & 36    & 1.801 & 6.64  &       &       & 45 \\
          & 46    & 9     & [9,0,0,0,0,0,0,0,0,0] & 39    & 1.66  & 10.32 &       &       & 48 \\
          & 58    & 10    & [10,0,0,0,0,0,0,0,0,0] & 38    & 1.69  & 12.23 &       &       & 50 \\
          & 89    & 9     & [7,0,1,0,0,0,0,1,0,0] & 47    & 2.681 & 17.33 &       &       & 68 \\
          & 108   & 15    & [15,0,0,0,0,0,0,0,0,0] & 63    & 3.278 & 24    &       &       & 83 \\
    S10(2) & 114   & 16    & [16,0,0,0,0,0,0,0,0,0] & 57    & 3.07  & 28.81 & 35.87 & 3.58  & 78 \\
          & 117   & 15    & [15,0,0,0,0,0,0,0,0,0] & 63    & 4.316 & 21.65 &       &       & 88 \\
          & 126   & 16    & [16,0,0,0,0,0,0,0,0,0] & 73    & 4.243 & 22.48 &       &       & 99 \\
          & 130   & 22    & [22,0,0,0,0,0,0,0,0,0] & 64    & 3.399 & 38.59 &       &       & 98 \\
          & 173   & 23    & [23,0,0,0,0,0,0,0,0,0] & 107   & 9.73  & 80.49 &       &       & 161 \\\hline
          & 20    & 4     & [0,0,0,4,0,0,0,0,0,0] & 39    & 1.48  & 5.1   &       &       & 43 \\
          & 22    & 5     & [5,0,0,0,0,0,0,0,0,0] & 43    & 1.59  & 5.41  &       &       & 45 \\
          & 24    & 5     & [5,0,0,0,0,0,0,0,0,0] & 36    & 1.3   & 5.41  &       &       & 45 \\
          & 26    & 5     & [3,2,0,0,0,0,0,0,0,0] & 33    & 1.64  & 3.49  &       &       & 44 \\
    S10(3) & 54    & 6     & [0,6,0,0,0,0,0,0,0,0] & 52    & 2.77  & 14.05 & 99.49 & 9.94  & 88 \\
          & 77    & 9     & [7,0,2,0,0,0,0,0,0,0] & 93    & 13.27 & 26.72 &       &       & 176 \\
          & 83    & 9     & [6,0,2,1,0,0,0,0,0,0] & 109   & 19.41 & 33.83 &       &       & 198 \\
          & 89    & 10    & [10,0,0,0,0,0,0,0,0,0] & 100   & 13.85 & 41.18 &       &       & 219 \\
          & 93    & 10    & [10,0,0,0,0,0,0,0,0,0] & 109   & 21.75 & 46.17 &       &       & 254 \\
          & 95    & 10    & [10,0,0,0,0,0,0,0,0,0] & 114   & 22.4  & 52.46 &       &       & 251 \\\hline
          & 131   & 7     & [5,0,0,0,2,0,0,0,0,0] & 63    & 8.34  & 61.44 &       &       & 102 \\
          & 136   & 6     & [3,1,0,0,2,0,0,0,0,0] & 47    & 7.23  & 54.38 &       &       & 88 \\
          & 171   & 6     & [2,2,0,0,2,0,0,0,0,0] & 65    & 11.18 & 56.92 &       &       & 102 \\
          & 173   & 7     & [3,1,3,0,0,0,0,0,0,0] & 60    & 9.87  & 116.22 &       &       & 118 \\
    S10(4) & 239   & 8     & [5,2,0,0,0,0,0,1,0,0] & 83    & 16.81 & 104.66 & 225.55 & 22.55 & 155 \\
          & 278   & 10    & [10,0,0,0,0,0,0,0,0,0] & 80    & 14.93 & 129.1 &       &       & 208 \\
          & 287   & 10    & [10,0,0,0,0,0,0,0,0,0] & 62    & 11.628 & 128.1 &       &       & 178 \\
          & 364   & 10    & [7,3,0,0,0,0,0,0,0,0] & 128   & 34.053 & 227.12 &       &       & 260 \\
          & 483   & 11    & [9,1,0,0,0,0,1,0,0,0] & 204   & 105.146 & 497   &       &       & 427 \\\hline
          & 146   & 8     & [0,6,2,0,0,0,0,0,0,0] & 42    & 8.048 & 100.82 &       &       & 70 \\
          & 173   & 10    & [10,0,0,0,0,0,0,0,0,0] & 71    & 15.43 & 115.39 &       &       & 99 \\
          & 207   & 10    & [10,0,0,0,0,0,0,0,0,0] & 60    & 11.77 & 138.87 &       &       & 82 \\
          & 359   & 12    & [7,5,0,0,0,0,0,0,0,0] & 60    & 14.69 & 198.246 &       &       & 152 \\
    S10(5) & 426   & 12    & [12,0,0,0,0,0,0,0,0,0] & 77    & 16.23 & 290.08 & 315.14 & 31.51 & 130 \\
          & 548   & 12    & [0,12,0,0,0,0,0,0,0,0] & 105   & 38.525 & 470.76 &       &       & 209 \\
          & 604   & 15    & [15,0,0,0,0,0,0,0,0,0] & 124   & 43.81 & 499.9 &       &       & 244 \\
          & 606   & 13    & [13,0,0,0,0,0,0,0,0,0] & 98    & 28.4  & 422.96 &       &       & 243 \\
          & 657   & 12    & [0,8,4,0,0,0,0,0,0,0] & 105   & 65.01 & 558.71 &       &       & 244 \\
          & 702   & 14    & [14,0,0,0,0,0,0,0,0,0] & 159   & 73.19 & 718.58 &       &       & 362 \\\hline
\caption{Results of computational experiments for $p=5, 10$.}
\label{table:2}
    \end{longtable}}
{\small
  \centering
    \begin{longtable}{|c|c|c|>{\small}c|c|c|c|c|c|c|}
\hline
    \texttt{S}     & $\texttt{n}_j$  & $\omega(\texttt{S}, \texttt{n}_j)$ & \texttt{min}   & \texttt{it}  & $\texttt{time}_j$    & \texttt{GAPtime} & \texttt{tottime} & \texttt{avtime} & $\#\texttt{min}$\\\hline
           & 47    & 6     & [6,0,0,0,0,0,0,0,0,0,0,0,0,0,0] & 114   & 8.386 & 13.78 &       &       & 129 \\
          & 65    & 5     & [4,0,0,1,0,0,0,0,0,0,0,0,0,0,0] & 112   & 10.358 & 35.64 &       &       & 159 \\
          & 79    & 5     & [3,1,0,0,1,0,0,0,0,0,0,0,0,0,0] & 105   & 9.392 & 7.21  &       &       & 165 \\
          & 82    & 6     & [0,6,0,0,0,0,0,0,0,0,0,0,0,0,0] & 141   & 13.664 & 17.93 &       &       & 184 \\
          & 84    & 6     & [4,1,0,0,1,0,0,0,0,0,0,0,0,0,0] & 112   & 11.156 & 28.24 &       &       & 192 \\
          & 91    & 7     & [7,0,0,0,0,0,0,0,0,0,0,0,0,0,0] & 101   & 6.863 & 9.34  &       &       & 173 \\
    S15(1) & 96    & 7     & [4,3,0,0,0,0,0,0,0,0,0,0,0,0,0] & 104   & 11.98 & 52.225 & 612.099 & 40.8066 & 250 \\
          & 100   & 8     & [8,0,0,0,0,0,0,0,0,0,0,0,0,0,0] & 187   & 26.425 & 29.725 &       &       & 251 \\
          & 109   & 7     & [7,0,0,0,0,0,0,0,0,0,0,0,0,0,0] & 129   & 12.659 & 35.725 &       &       & 245 \\
          & 121   & 8     & [8,0,0,0,0,0,0,0,0,0,0,0,0,0,0] & 154   & 19.271 & 48.225 &       &       & 307 \\
          & 124   & 8     & [8,0,0,0,0,0,0,0,0,0,0,0,0,0,0] & 214   & 37.796 & 203.8 &       &       & 364 \\
          & 134   & 7     & [5,1,1,0,0,0,0,0,0,0,0,0,0,0,0] & 168   & 29.652 & 241.9 &       &       & 383 \\
          & 139   & 8     & [8,0,0,0,0,0,0,0,0,0,0,0,0,0,0] & 183   & 30.122 & 199.05 &       &       & 394 \\
          & 169   & 8     & [5,2,0,0,0,1,0,0,0,0,0,0,0,0,0] & 285   & 38.405 & 164.625 &       &       & 680 \\\hline
          & 46    & 5     & [0,1,3,0,1,0,0,0,0,0,0,0,0,0,0] & 83    & 8.383 & 79.4  &       &       & 98 \\
          & 115   & 6     & [2,0,3,0,1,0,0,0,0,0,0,0,0,0,0] & 94    & 10.454 & 112.65 &       &       & 109 \\
          & 155   & 17    & [17,0,0,0,0,0,0,0,0,0,0,0,0,0,0] & 123   & 14.728 & 139.425 &       &       & 151 \\
          & 286   & 15    & [12,0,3,0,0,0,0,0,0,0,0,0,0,0,0] & 137   & 20.015 & 291.65 &       &       & 206 \\
          & 289   & 15    & [15,0,0,0,0,0,0,0,0,0,0,0,0,0,0] & 109   & 14.545 & 293.575 &       &       & 190 \\
          & 341   & 17    & [17,0,0,0,0,0,0,0,0,0,0,0,0,0,0] & 174   & 65.975 & 401   &       &       & 252 \\
          & 342   & 15    & [14,0,0,0,1,0,0,0,0,0,0,0,0,0,0] & 192   & 32.986 & 406.775 &       &       & 265 \\
    S15(2) & 348   & 15    & [13,0,2,0,0,0,0,0,0,0,0,0,0,0,0] & 193   & 113.383 & 427.3 & 1683.63 & 112.242 & 291 \\
          & 393   & 20    & [20,0,0,0,0,0,0,0,0,0,0,0,0,0,0] & 228   & 135.869 & 550.35 &       &       & 320 \\
          & 436   & 25    & [25,0,0,0,0,0,0,0,0,0,0,0,0,0,0] & 273   & 74.036 & 736.575 &       &       & 413 \\
          & 445   & 24    & [24,0,0,0,0,0,0,0,0,0,0,0,0,0,0] & 311   & 96.198 & 784.625 &       &       & 434 \\
          & 449   & 19    & [19,0,0,0,0,0,0,0,0,0,0,0,0,0,0] & 294   & 82.945 & 795.45 &       &       & 425 \\
          & 504   & 22    & [22,0,0,0,0,0,0,0,0,0,0,0,0,0,0] & 354   & 177.154 & 1161.45 &       &       & 594 \\
          & 527   & 20    & [20,0,0,0,0,0,0,0,0,0,0,0,0,0,0] & 367   & 166.69 & 1345.275 &       &       & 610 \\
          & 584   & 26    & [26,0,0,0,0,0,0,0,0,0,0,0,0,0,0] & 438   & 670.269 & 1988.875 &       &       & 737 \\\hline
          & 40    & 3     & [0,1,0,0,1,0,0,0,0,1,0,0,0,0,0] & 89    & 7.53  & 24.85 &       &       & 113 \\
          & 84    & 8     & [8,0,0,0,0,0,0,0,0,0,0,0,0,0,0] & 119   & 11.648 & 39.95 &       &       & 139 \\
          & 126   & 10    & [10,0,0,0,0,0,0,0,0,0,0,0,0,0,0] & 114   & 13.082 & 67.375 &       &       & 195 \\
          & 130   & 10    & [10,0,0,0,0,0,0,0,0,0,0,0,0,0,0] & 117   & 12.165 & 68.425 &       &       & 188 \\
          & 132   & 10    & [10,0,0,0,0,0,0,0,0,0,0,0,0,0,0] & 126   & 12.849 & 69.1  &       &       & 186 \\
          & 135   & 10    & [10,0,0,0,0,0,0,0,0,0,0,0,0,0,0] & 149   & 23.85 & 72.4  &       &       & 212 \\
          & 142   & 10    & [10,0,0,0,0,0,0,0,0,0,0,0,0,0,0] & 156   & 20.35 & 78.35 &       &       & 216 \\
    S15(3) & 152   & 8     & [8,0,0,0,0,0,0,0,0,0,0,0,0,0,0] & 176   & 22.83 & 90.65 & 1407.684 & 93.8456 & 267 \\
          & 165   & 12    & [12,0,0,0,0,0,0,0,0,0,0,0,0,0,0] & 183   & 31.199 & 109.15 &       &       & 305 \\
          & 183   & 10    & [10,0,0,0,0,0,0,0,0,0,0,0,0,0,0] & 189   & 38.537 & 134.025 &       &       & 324 \\
          & 217   & 11    & [9,1,0,0,1,0,0,0,0,0,0,0,0,0,0] & 244   & 77.582 & 219.075 &       &       & 446 \\
          & 221   & 11    & [11,0,0,0,0,0,0,0,0,0,0,0,0,0,0] & 250   & 70.445 & 229.625 &       &       & 454 \\
          & 229   & 13    & [13,0,0,0,0,0,0,0,0,0,0,0,0,0,0] & 274   & 89.679 & 253.775 &       &       & 478 \\
          & 273   & 14    & [14,0,0,0,0,0,0,0,0,0,0,0,0,0,0] & 353   & 133.473 & 475.725 &       &       & 701 \\
          & 323   & 15    & [15,0,0,0,0,0,0,0,0,0,0,0,0,0,0] & 578   & 842.465 & 3487.55 &       &       & 1076 \\\hline
          & 75    & 5     & [0,4,0,1,0,0,0,0,0,0,0,0,0,0,0] & 92    & 8.229 & 42.65 &       &       & 123 \\
          & 104   & 6     & [4,0,0,1,1,0,0,0,0,0,0,0,0,0,0] & 96    & 10.708 & 52.2  &       &       & 140 \\
          & 114   & 7     & [5,0,0,2,0,0,0,0,0,0,0,0,0,0,0] & 86    & 8.044 & 54.875 &       &       & 144 \\
          & 128   & 7     & [5,0,2,0,0,0,0,0,0,0,0,0,0,0,0] & 104   & 10.528 & 61.075 &       &       & 147 \\
          & 216   & 8     & [3,0,5,0,0,0,0,0,0,0,0,0,0,0,0] & 151   & 23.513 & 126.85 &       &       & 254 \\
          & 219   & 9     & [8,1,0,0,0,0,0,0,0,0,0,0,0,0,0] & 145   & 19.017 & 126.95 &       &       & 251 \\
          & 241   & 9     & [9,0,0,0,0,0,0,0,0,0,0,0,0,0,0] & 139   & 22.934 & 156.65 &       &       & 302 \\
    S15(4) & 271   & 8     & [4,3,1,0,0,0,0,0,0,0,0,0,0,0,0] & 188   & 52.956 & 202.2 & 892.038 & 59.4692 & 357 \\
          & 309   & 9     & [6,1,1,1,0,0,0,0,0,0,0,0,0,0,0] & 211   & 62.039 & 284.25 &       &       & 439 \\
          & 310   & 10    & [10,0,0,0,0,0,0,0,0,0,0,0,0,0,0] & 188   & 34.511 & 281.525 &       &       & 429 \\
          & 321   & 12    & [12,0,0,0,0,0,0,0,0,0,0,0,0,0,0] & 269   & 67.673 & 320.65 &       &       & 523 \\
          & 327   & 10    & [10,0,0,0,0,0,0,0,0,0,0,0,0,0,0] & 185   & 43.396 & 341.225 &       &       & 502 \\
          & 340   & 11    & [11,0,0,0,0,0,0,0,0,0,0,0,0,0,0] & 212   & 42.852 & 374.325 &       &       & 509 \\
          & 352   & 10    & [7,3,0,0,0,0,0,0,0,0,0,0,0,0,0] & 262   & 300.513 & 2434.3 &       &       & 622 \\
          & 371   & 11    & [11,0,0,0,0,0,0,0,0,0,0,0,0,0,0] & 315   & 185.125 & 505.925 &       &       & 661 \\\hline
          & 29    & 5     & [0,5,0,0,0,0,0,0,0,0,0,0,0,0,0] & 97    & 3.899 & 258.275 &       &       & 106 \\
          & 50    & 5     & [5,0,0,0,0,0,0,0,0,0,0,0,0,0,0] & 99    & 4.911 & 533.775 &       &       & 116 \\
          & 95    & 6     & [3,2,0,0,0,1,0,0,0,0,0,0,0,0,0] & 115   & 9.017 & 73.725 &       &       & 154 \\
          & 96    & 10    & [10,0,0,0,0,0,0,0,0,0,0,0,0,0,0] & 120   & 7.652 & 24.55 &       &       & 171 \\
          & 99    & 9     & [9,0,0,0,0,0,0,0,0,0,0,0,0,0,0] & 127   & 9.08  & 25.7  &       &       & 170 \\
          & 109   & 9     & [9,0,0,0,0,0,0,0,0,0,0,0,0,0,0] & 138   & 11.984 & 30.225 &       &       & 187 \\
          & 110   & 10    & [10,0,0,0,0,0,0,0,0,0,0,0,0,0,0] & 124   & 7.249 & 31.025 &       &       & 192 \\
    S15(5) & 119   & 11    & [11,0,0,0,0,0,0,0,0,0,0,0,0,0,0] & 127   & 9.411 & 36.375 & 685.921 & 45.72807 & 211 \\
          & 131   & 10    & [9,1,0,0,0,0,0,0,0,0,0,0,0,0,0] & 121   & 10.939 & 45.3  &       &       & 252 \\
          & 134   & 11    & [11,0,0,0,0,0,0,0,0,0,0,0,0,0,0] & 174   & 17.188 & 48.075 &       &       & 258 \\
          & 135   & 11    & [11,0,0,0,0,0,0,0,0,0,0,0,0,0,0] & 173   & 20.678 & 225.525 &       &       & 288 \\
          & 152   & 10    & [8,2,0,0,0,0,0,0,0,0,0,0,0,0,0] & 180   & 28.469 & 144.375 &       &       & 338 \\
          & 162   & 9     & [8,0,0,0,1,0,0,0,0,0,0,0,0,0,0] & 205   & 32.88 & 83.55 &       &       & 357 \\
          & 180   & 10    & [8,2,0,0,0,0,0,0,0,0,0,0,0,0,0] & 269   & 71.89 & 269.85 &       &       & 471 \\
          & 201   & 14    & [14,0,0,0,0,0,0,0,0,0,0,0,0,0,0] & 306   & 440.674 & 1883.525 &       &       & 584 \\\hline
          & 131   & 8     & [0,8,0,0,0,0,0,0,0,0,0,0,0,0,0,0,0,0,0,0] & 188   & 35.321 & 321.325 &       &       & 264 \\
          & 145   & 8     & [8,0,0,0,0,0,0,0,0,0,0,0,0,0,0,0,0,0,0,0] & 191   & 34.252 & 332.3 &       &       & 265 \\
          & 249   & 9     & [6,2,0,1,0,0,0,0,0,0,0,0,0,0,0,0,0,0,0,0] & 233   & 54.57 & 550.725 &       &       & 340 \\
          & 257   & 9     & [6,2,0,1,0,0,0,0,0,0,0,0,0,0,0,0,0,0,0,0] & 233   & 53.913 & 569.15 &       &       & 352 \\
          & 260   & 8     & [8,0,0,0,0,0,0,0,0,0,0,0,0,0,0,0,0,0,0,0] & 197   & 47.428 & 573.775 &       &       & 355 \\
          & 319   & 9     & [1,7,0,1,0,0,0,0,0,0,0,0,0,0,0,0,0,0,0,0] & 244   & 90.809 & 785.925 &       &       & 451 \\
          & 354   & 9     & [4,5,0,0,0,0,0,0,0,0,0,0,0,0,0,0,0,0,0,0] & 256   & 97.12 & 938.925 &       &       & 500 \\
          & 459   & 10    & [0,10,0,0,0,0,0,0,0,0,0,0,0,0,0,0,0,0,0,0] & 398   & 182.085 & 1700.525 &       &       & 787 \\
          & 465   & 9     & [9,0,0,0,0,0,0,0,0,0,0,0,0,0,0,0,0,0,0,0] & 356   & 363.725 & 1747.575 &       &       & 752 \\
    S20(1) & 469   & 11    & [3,8,0,0,0,0,0,0,0,0,0,0,0,0,0,0,0,0,0,0] & 317   & 239.548 & 1802.75 & 19603.67 & 980.1835 & 796 \\
          & 487   & 9     & [9,0,0,0,0,0,0,0,0,0,0,0,0,0,0,0,0,0,0,0] & 384   & 408.842 & 2011.8 &       &       & 865 \\
          & 572   & 12    & [6,6,0,0,0,0,0,0,0,0,0,0,0,0,0,0,0,0,0,0] & 399   & 826.33 & 3290.4 &       &       & 1160 \\
          & 575   & 10    & [0,10,0,0,0,0,0,0,0,0,0,0,0,0,0,0,0,0,0,0] & 542   & 3123.4 & 3414.425 &       &       & 1235 \\
          & 587   & 12    & [12,0,0,0,0,0,0,0,0,0,0,0,0,0,0,0,0,0,0,0] & 487   & 1356.94 & 3657.525 &       &       & 1273 \\
          & 606   & 11    & [11,0,0,0,0,0,0,0,0,0,0,0,0,0,0,0,0,0,0,0] & 507   & 2100.84 & 4119.5 &       &       & 1389 \\
          & 607   & 10    & [5,5,1,0,1,0,1,0,0,0,0,0,0,0,0,0,0,0,0,0] & 787   & 2894.21 & 4181.7 &       &       & 1387 \\
          & 652   & 11    & [9,0,0,0,0,3,0,0,0,0,0,0,0,0,0,0,0,0,0,0] & 683   & 2013.25 & 5538.2 &       &       & 1673 \\
          & 674   & 12    & [12,0,0,0,0,0,0,0,0,0,0,0,0,0,0,0,0,0,0,0] & 821   & 1999.747 & 6288.875 &       &       & 1732 \\
          & 694   & 11    & [11,0,0,0,0,0,0,0,0,0,0,0,0,0,0,0,0,0,0,0] & 726   & 2991.08 & 7185.075 &       &       & 1851 \\
          & 762   & 12    & [12,0,0,0,0,0,0,0,0,0,0,0,0,0,0,0,0,0,0,0] & 118   & 690.26 & 11142.2 &       &       & 2375 \\\hline
          & 57    & 4     & [0,2,0,0,1,0,0,0,0,0,0,0,0,1,0,0,0,0,0,0] & 146   & 15.003 & 108.725 &       &       & 186 \\
          & 105   & 7     & [7,0,0,0,0,0,0,0,0,0,0,0,0,0,0,0,0,0,0,0] & 169   & 20.95 & 158.525 &       &       & 211 \\
          & 182   & 9     & [6,3,0,0,0,0,0,0,0,0,0,0,0,0,0,0,0,0,0,0] & 195   & 31.558 & 311.125 &       &       & 304 \\
          & 186   & 11    & [11,0,0,0,0,0,0,0,0,0,0,0,0,0,0,0,0,0,0,0] & 272   & 61.55 & 328.45 &       &       & 364 \\
          & 201   & 14    & [14,0,0,0,0,0,0,0,0,0,0,0,0,0,0,0,0,0,0,0] & 319   & 300.1 & 374.2 &       &       & 409 \\
          & 204   & 9     & [9,0,0,0,0,0,0,0,0,0,0,0,0,0,0,0,0,0,0,0] & 284   & 63.721 & 394.55 &       &       & 427 \\
          & 254   & 9     & [8,0,0,0,1,0,0,0,0,0,0,0,0,0,0] & 378   & 153.266 & 635.725 &       &       & 599 \\
          & 259   & 10    & [7,3,0,0,0,0,0,0,0,0,0,0,0,0,0,0,0,0,0,0] & 295   & 72.789 & 653.9 &       &       & 530 \\
          & 263   & 10    & [6,4,0,0,0,0,0,0,0,0,0,0,0,0,0,0,0,0,0,0] & 294   & 69.81 & 695.8 &       &       & 612 \\
    S20(2) & 274   & 9     & [4,5,0,0,0,0,0,0,0,0,0,0,0,0,0,0,0,0,0,0] & 336   & 172.839 & 751.05 & 11703.9 & 585.1952 & 587 \\
          & 275   & 11    & [8,3,0,0,0,0,0,0,0,0,0,0,0,0,0,0,0,0,0,0] & 414   & 217.261 & 798.075 &       &       & 702 \\
          & 294   & 8     & [5,1,1,0,1,0,0,0,0,0,0,0,0,0,0,0,0,0,0,0] & 359   & 396.194 & 915.575 &       &       & 612 \\
          & 295   & 11    & [8,3,0,0,0,0,0,0,0,0,0,0,0,0,0,0,0,0,0,0] & 488   & 2342.7 & 975.65 &       &       & 802 \\
          & 298   & 12    & [12,0,0,0,0,0,0,0,0,0,0,0,0,0,0,0,0,0,0,0] & 463   & 174.03 & 990.675 &       &       & 773 \\
          & 307   & 10    & [6,4,0,0,0,0,0,0,0,0,0,0,0,0,0,0,0,0,0,0] & 391   & 187.431 & 1084.4 &       &       & 808 \\
          & 338   & 13    & [13,0,0,0,0,0,0,0,0,0,0,0,0,0,0,0,0,0,0,0] & 484   & 432.837 & 1546.65 &       &       & 1001 \\
          & 367   & 14    & [14,0,0,0,0,0,0,0,0,0,0,0,0,0,0,0,0,0,0,0] & 608   & 2000.37 & 2113.475 &       &       & 1248 \\
          & 393   & 12    & [3,9,0,0,0,0,0,0,0,0,0,0,0,0,0,0,0,0,0,0] & 658   & 2531.045 & 2889.325 &       &       & 1502 \\
          & 417   & 11    & [5,0,2,0,0,0,0,0,0,2,3,0,0,0,0,0,0,0,0,0] & 563   & 1093.34 & 3737.325 &       &       & 1686 \\
          & 431   & 14    & [6,8,0,0,0,0,0,0,0,0,0,0,0,0,0,0,0,0,0,0] & 781   & 1367.11 & *     &       &       & * \\\hline
          & 85    & 4     & [0,4,0,0,0,0,0,0,0,0,0,0,0,0,0,0,0,0,0,0] & 147   & 25.646 & 341.175 &       &       & 230 \\
          & 298   & 15    & [15,0,0,0,0,0,0,0,0,0,0,0,0,0,0,0,0,0,0,0] & 316   & 99.513 & 864.5 &       &       & 455 \\
          & 333   & 15    & [15,0,0,0,0,0,0,0,0,0,0,0,0,0,0,0,0,0,0,0] & 322   & 91.273 & 1007.75 &       &       & 465 \\
          & 342   & 16    & [16,0,0,0,0,0,0,0,0,0,0,0,0,0,0,0,0,0,0,0] & 326   & 99.809 & 1026.075 &       &       & 466 \\
          & 349   & 16    & [16,0,0,0,0,0,0,0,0,0,0,0,0,0,0,0,0,0,0,0] & 307   & 76.393 & 1075.375 &       &       & 512 \\
          & 358   & 15    & [15,0,0,0,0,0,0,0,0,0,0,0,0,0,0,0,0,0,0,0] & 324   & 86.003 & 1092.975 &       &       & 480 \\
          & 401   & 12    & [10,0,0,0,0,2,0,0,0,0,0,0,0,0,0,0,0,0,0,0] & 394   & 154.683 & 1372.975 &       &       & 631 \\
          & 415   & 16    & [16,0,0,0,0,0,0,0,0,0,0,0,0,0,0,0,0,0,0,0] & 448   & 293.327 & 1474  &       &       & 687 \\
          & 462   & 15    & [15,0,0,0,0,0,0,0,0,0,0,0,0,0,0,0,0,0,0,0] & 361   & 135.922 & 1827.75 &       &       & 691 \\
    S20(3) & 480   & 14    & [14,0,0,0,0,0,0,0,0,0,0,0,0,0,0,0,0,0,0,0] & 527   & 261.581 & 1982.85 & 8912.075 & 445.6038 & 786 \\
          & 556   & 16    & [16,0,0,0,0,0,0,0,0,0,0,0,0,0,0,0,0,0,0,0] & 610   & 592.666 & 2975.05 &       &       & 1028 \\
          & 569   & 18    & [18,0,0,0,0,0,0,0,0,0,0,0,0,0,0,0,0,0,0,0] & 695   & 2453.98 & 3284.75 &       &       & 1158 \\
          & 583   & 19    & [19,0,0,0,0,0,0,0,0,0,0,0,0,0,0,0,0,0,0,0] & 711   & 1496.19 & 3440.55 &       &       & 1164 \\
          & 609   & 15    & [15,0,0,0,0,0,0,0,0,0,0,0,0,0,0,0,0,0,0,0] & 57    & 11.671 & 4037.25 &       &       & 1290 \\
          & 619   & 13    & [12,0,0,0,0,0,0,0,0,0,0,0,0,0,0,0,0,0,0,0] & 912   & 990.061 & 4518.725 &       &       & 1386 \\
          & 708   & 18    & [18,0,0,0,0,0,0,0,0,0,0,0,0,0,0,0,0,0,0,0] & 582   & 569.397 & *     &       &       & * \\
          & 710   & 15    & [11,3,0,0,0,0,0,0,0,0,0,0,0,0,0,0,0,0,0,0] & 733   & 673.064 & *     &       &       & * \\
          & 752   & 18    & [18,0,0,0,0,0,0,0,0,0,0,0,0,0,0,0,0,0,0,0] & 777   & 722.174 & *     &       &       & * \\
          & 821   & 21    & [18,0,0,0,0,0,1,0,0,0,0,0,0,0,0,0,0,0,0,0] & 108   & 35.333 & *     &       &       & * \\
          & 853   & 21    & [19,0,0,0,0,1,0,0,0,0,0,0,0,0,0,0,0,0,0,0] & 108   & 43.389 & *     &       &       & * \\\hline
          & 81    & 5     & [0,5,0,0,0,0,0,0,0,0,0,0,0,0,0,0,0,0,0,0] & 140   & 17.706 & 219.175 &       &       & 214 \\
          & 107   & 6     & [6,0,0,0,0,0,0,0,0,0,0,0,0,0,0,0,0,0,0,0] & 176   & 27.955 & 264.375 &       &       & 251 \\
          & 168   & 9     & [7,2,0,0,0,0,0,0,0,0,0,0,0,0,0,0,0,0,0,0] & 185   & 26.037 & 416.725 &       &       & 291 \\
          & 194   & 9     & [6,3,0,0,0,0,0,0,0,0,0,0,0,0,0,0,0,0,0,0] & 239   & 57.189 & 527.75 &       &       & 427 \\
          & 230   & 8     & [4,4,0,0,0,0,0,0,0,0,0,0,0,0,0,0,0,0,0,0] & 186   & 39    & 707.575 &       &       & 471 \\
          & 236   & 9     & [7,2,0,0,0,0,0,0,0,0,0,0,0,0,0,0,0,0,0,0] & 198   & 49.312 & 735.875 &       &       & 474 \\
          & 274   & 9     & [8,0,1,0,0,0,0,0,0,0,0,0,0,0,0,0,0,0,0,0] & 284   & 524.051 & 1027.225 &       &       & 590 \\
          & 277   & 9     & [7,2,0,0,0,0,0,0,0,0,0,0,0,0,0,0,0,0,0,0] & 276   & 113.266 & 1079.7 &       &       & 679 \\
          & 286   & 9     & [6,2,1,0,0,0,0,0,0,0,0,0,0,0,0,0,0,0,0,0] & 321   & 676.315 & 1143.675 &       &       & 698 \\
    S20(4) & 290   & 8     & [1,7,0,0,0,0,0,0,0,0,0,0,0,0,0,0,0,0,0,0] & 312   & 1775.88 & 1177.15 & 11215.3 & 560.7652 & 630 \\
          & 305   & 10    & [7,3,0,0,0,0,0,0,0,0,0,0,0,0,0,0,0,0,0,0] & 256   & 188.763 & 1345.125 &       &       & 683 \\
          & 310   & 10    & [10,0,0,0,0,0,0,0,0,0,0,0,0,0,0,0,0,0,0,0] & 297   & 85.818 & 1407.6 &       &       & 704 \\
          & 348   & 10    & [7,3,0,0,0,0,0,0,0,0,0,0,0,0,0,0,0,0,0,0] & 403   & 392.226 & 2039.675 &       &       & 953 \\
          & 351   & 11    & [11,0,0,0,0,0,0,0,0,0,0,0,0,0,0,0,0,0,0,0] & 432   & 193.712 & 2079.4 &       &       & 949 \\
          & 366   & 10    & [9,0,0,1,0,0,0,0,0,0,0,0,0,0,0,0,0,0,0,0] & 383   & 295.208 & 2346.175 &       &       & 912 \\
          & 379   & 10    & [3,7,0,0,0,0,0,0,0,0,0,0,0,0,0,0,0,0,0,0] & 296   & 1007.8 & 2735.2 &       &       & 1116 \\
          & 396   & 11    & [10,1,0,0,0,0,0,0,0,0,0,0,0,0,0,0,0,0,0,0] & 560   & 3095.66 & 3283.85 &       &       & 1222 \\
          & 416   & 12    & [12,0,0,0,0,0,0,0,0,0,0,0,0,0,0,0,0,0,0,0] & 541   & 693.549 & *     &       &       & * \\
          & 521   & 14    & [14,0,0,0,0,0,0,0,0,0,0,0,0,0,0,0,0,0,0,0] & 611   & 955.771 & *     &       &       & * \\
          & 583   & 14    & [14,0,0,0,0,0,0,0,0,0,0,0,0,0,0,0,0,0,0,0] & 982   & 1000.085 & *     &       &       & * \\\hline
          & 101   & 8     & [0,8,0,0,0,0,0,0,0,0,0,0,0,0,0,0,0,0,0,0] & 57    & 8.393 & 488.75 &       &       & 246 \\
          & 141   & 7     & [7,0,0,0,0,0,0,0,0,0,0,0,0,0,0,0,0,0,0,0] & 146   & 33.759 & 567.975 &       &       & 260 \\
          & 279   & 12    & [12,0,0,0,0,0,0,0,0,0,0,0,0,0,0,0,0,0,0,0] & 99    & 23.478 & 1170.65 &       &       & 502 \\
          & 314   & 10    & [10,0,0,0,0,0,0,0,0,0,0,0,0,0,0,0,0,0,0,0] & 199   & 68.421 & 1328.55 &       &       & 457 \\
          & 329   & 11    & [7,4,0,0,0,0,0,0,0,0,0,0,0,0,0,0,0,0,0,0] & 85    & 17.706 & 1428.15 &       &       & 461 \\
          & 369   & 11    & [5,5,0,1,0,0,0,0,0,0,0,0,0,0,0,0,0,0,0,0] & 106   & 25.178 & 1747.875 &       &       & 493 \\
          & 399   & 11    & [7,4,0,0,0,0,0,0,0,0,0,0,0,0,0,0,0,0,0,0] & 156   & 90.957 & 2166.55 &       &       & 711 \\
          & 425   & 11    & [5,6,0,0,0,0,0,0,0,0,0,0,0,0,0,0,0,0,0,0] & 65    & 26.208 & 2477.4 &       &       & 718 \\
          & 438   & 13    & [13,0,0,0,0,0,0,0,0,0,0,0,0,0,0,0,0,0,0,0] & 115   & 37.799 & 2648.8 &       &       & 732 \\
    S20(5) & 447   & 15    & [15,0,0,0,0,0,0,0,0,0,0,0,0,0,0,0,0,0,0,0] & 60    & 11.342 & 2771.675 & 10007.05 & 500.3524 & 808 \\
          & 477   & 14    & [14,0,0,0,0,0,0,0,0,0,0,0,0,0,0,0,0,0,0,0] & 138   & 51.901 & 3357.7 &       &       & 929 \\
          & 501   & 16    & [16,0,0,0,0,0,0,0,0,0,0,0,0,0,0,0,0,0,0,0] & 40    & 7.425 & 3884.325 &       &       & 1026 \\
          & 534   & 12    & [12,0,0,0,0,0,0,0,0,0,0,0,0,0,0,0,0,0,0,0] & 180   & 145.075 & 4557.05 &       &       & 983 \\
          & 536   & 14    & [14,0,0,0,0,0,0,0,0,0,0,0,0,0,0,0,0,0,0,0] & 83    & 23.15 & 4752.25 &       &       & 1090 \\
          & 555   & 13    & [9,3,0,1,0,0,0,0,0,0,0,0,0,0,0,0,0,0,0,0] & 166   & 63.804 & 5404.225 &       &       & 1190 \\
          & 574   & 13    & [13,0,0,0,0,0,0,0,0,0,0,0,0,0,0,0,0,0,0,0] & 345   & 777.094 & *     &       &       & * \\
          & 620   & 14    & [14,0,0,0,0,0,0,0,0,0,0,0,0,0,0,0,0,0,0,0] & 721   & 654.063 & *     &       &       & * \\
          & 727   & 18    & [18,0,0,0,0,0,0,0,0,0,0,0,0,0,0,0,0,0,0,0] & 882   & 2140.435 & *     &       &       & * \\
          & 786   & 17    & [17,0,0,0,0,0,0,0,0,0,0,0,0,0,0,0,0,0,0,0] & 734   & 3000.11 & *     &       &       & * \\
          & 871   & 17    & [17,0,0,0,0,0,0,0,0,0,0,0,0,0,0,0,0,0,0,0] & 813   & 2800.75 & *     &       &       & * \\\hline
\hline
    \caption{Results of computational experiments for $p=15, 20$.}
\label{table:3}
    \end{longtable}}

From the computational tests we can assert that our method improves the current algorithm implemented in GAP used to compute the $\omega$ invariant of a numerical semigroup because several reasons:
\begin{itemize}
\item The CPU times of our algorithm are much smaller than the brute force implementation in GAP. Observe that the average quotient between $\texttt{time}_j$ and $\texttt{GAPtime}$ is $0.23$.

\item Our algorithm avoids the complete enumeration of the set of non-dominated solutions of $({\rm SMIP}_j)$. Note that the average percentage of explored solutions ($\texttt{it}_j$) with respect to the number of non-dominated solutions ($\#\texttt{min}$) is $58.9\%$ while in GAP the 100\% of the solutions are explored. The algorithm proposed in \cite{AChKT10} also needs to compute a large number of solutions to obtain the $\omega$ invariant, so although we have not explicitly compared the running times with this method, our algorithm only searches a part of the efficient solutions and then, it must take much smaller CPU times.
\item Algorithm \ref{alg:1} is able to compute the $\omega$ invariant of the whole battery of experiments while the GAP method was not able to solve 14 problems (*).
    \end{itemize}

\section{Conclusions}

In this paper we propose a new methodology, based on discrete optimization tools, to compute a specific arithmetic invariant for numerical semigroups. We use a result in \cite{B-GS-G10} (Theorem \ref{theo:omega})  to translate the algebraic problem to the problem of optimizing a linear function over the non-dominated solutions of a multiobjective integer linear program that is basically a knapsack-type problem. Then, an adaptation of the methodology in  \cite{jorge} is applied to compute more efficiently than the current methods, the $\omega$ invariant of a numerical semigroup. The key here is to avoid the complete enumeration of the non-dominated solutions of the multiobjective problem to solve the problem.

The algorithm may be improved by using better bounds or by adding more information to the problem. It is here where bidirectional improvements in both fields, optimization and commutative algebra, may help to have better computational results.

Note that the proposed methodology can be adapted conveniently to compute the $\omega$ invariant of some other commutative semigroups. In particular, in the case when the semigroup is defined by a system of linear Diophantine equations, like in full and affine semigroups (also called saturated semigroups), from Theorem \ref{theo:omega} a similar result than the one in Theorem \ref{theo:1} can be derived but where the feasible region consists of several equations instead of the knapsack type equation. For those semigroups, can also be adapted Algorithm \ref{alg:1} to compute the $\omega$ invariant.

Furthermore, some other arithmetic invariant, as the tame degree, should also be modeled as a discrete optimization problem, and optimization techniques can help to solve more efficiently these problems. The analysis of different arithmetic invariant by discrete optimization is left for further research. Different approaches but also using tools from operational research may be used to solve problems for numerical semigroups such as the computation of minimal decomposition of a numerical semigroup into irreducible numerical semigroups or to define new indices that can be defined as optimal values of a linear objective function over the Kunz's polytope.

 We hope that this paper is the first of many bidirectional collaborations between the fields of algebra and operation research. The two fields seem to be very different from each other, but that may help to get new results in both areas and to understand with different viewpoints notions that usually live in one of them.

}

\end{document}